\documentclass[12pt]{article}
\usepackage{amsmath,amsthm,amssymb}
\usepackage{amssymb,latexsym}

\newtheorem{theorem}{Theorem}
\newtheorem{conjecture}{Conjecture}
\newtheorem{lemma}{Lemma}

\begin{document}

\baselineskip=17pt

\title{\bf On an logarithmic equation by primes}

\author{\bf S. I. Dimitrov}
\date{2019}
\maketitle
\begin{abstract}
Let $[\, \cdot\,]$ be the floor function.
In this paper we show that every sufficiently large positive integer $N$ can be represented in the form
\begin{equation*}
N=[p_1\log p_1]+[p_2\log p_2]+[p_3\log p_3],
\end{equation*}
where $p_1,\, p_2,\, p_3$ are prime numbers.
We also establish an asymptotic formula for the number of such representations,
when $p_1,\, p_2,\, p_3$ do not exceed given sufficiently large positive number. \\
\quad\\
\textbf{Keywords}: Diophantine equation $\cdot$ Logarithmic equation $\cdot$  Primes\\
\quad\\
{\bf  2010 Math.\ Subject Classification}:  11P32 $\cdot$ 11P55
\end{abstract}

\section{Introduction and main result}
\indent

A remarkable moment in analytic number theory is 1937,
when Vinogradov \cite{Vinogradov1} proved the ternary Goldbach problem.
He showed that every sufficiently large odd integer $N$ can be represented  in the form
\begin{equation*}
N=p_1+p_2+p_3,
\end{equation*}
where $p_1,\, p_2,\, p_3$ are prime numbers.

The consequences of Vinogradov's \cite{Vinogradov2} ingenious method for estimating
exponential sums over primes continue to this day in analytic number theory

In 1995 Laporta and Tolev  \cite{Laporta-Tolev} investigated an analogue of the Goldbach-Vinogradov theorem.
They considered the diophantine equation
\begin{equation*}
N=[p^c_1]+[p^c_2]+[p^c_3]\,,
\end{equation*}
where $p_1,\, p_2,\, p_3$ are primes.
For $1<c <17/16$ they showed that for the sum
\begin{equation*}
R(N)=\sum\limits_{N=[p^c_1]+[p^c_2]+[p^c_3]}\log p_1\log p_2\log p_3
\end{equation*}
the asymptotic formula
\begin{equation}\label{RNasymptoticformula}
R(N)=\frac{\Gamma^3(1 + 1/c)}{\Gamma(3/c)}N^{3/c-1}
+\mathcal{O}\Big(N^{3/c-1}\exp\big(-(\log N)^{1/3-\varepsilon}\big)\Big)
\end{equation}
holds.

Subsequently the result of Laporta and Tolev was sharpened by
Kumchev and Nedeva \cite{Kumchev-Nedeva} to
\begin{equation*}
1<c <\frac{12}{11},
\end{equation*}
by Zhai and Cao  \cite{Zhai-Cao} to
\begin{equation*}
1<c<\frac{258}{235},
\end{equation*}
by Cai \cite{Cai} to
\begin{equation*}
1<c <\frac{137}{119}.
\end{equation*}
Overcoming all difficulties Zhang and Li \cite{Zhang-Li3} improved the result of Cai to
\begin{equation*}
1<c <\frac{3113}{2703}
\end{equation*}
and this is the best result up to now.

On the other hand recently the author \cite{Dimitrov} showed that
when $N$ is a sufficiently large positive number and $\varepsilon>0$
is a small constant then the logarithmic inequality
\begin{equation*}
\big|p_1\log p_1+p_2\log p_2+p_3\log p_3-N\big|<\varepsilon,
\end{equation*}
has a solution in prime numbers $p_1,\,p_2,\,p_3$.

Motivated by these results in this paper we
introduce new diophantine equation with prime numbers.

Consider the logarithmic equation
\begin{equation}\label{Myequation}
N=[p_1\log p_1]+[p_2\log p_2]+[p_3\log p_3],
\end{equation}
where $N$ is a sufficiently large positive integer.
Having  the arguments of the aforementioned marvellous mathematicians and \cite{Dimitrov}
we expect that \eqref{Myequation} has a solutions in primes $p_1,\,p_2,\,p_3$.
Define the sum
\begin{equation}\label{Gamma}
\Gamma= \sum\limits_{N=[p_1\log p_1]+[p_2\log p_2]+[p_3\log p_3]}\log p_1\log p_2\log p_3\,.
\end{equation}
We make the first attempt and prove the following theorem.
\begin{theorem} Let $N$ is a sufficiently large positive integer.
Let $X$ is a solution of the equality
\begin{equation*}
X\log X=N.
\end{equation*}
Then the asymptotic formula
\begin{equation}\label{Asymptoticformula}
\Gamma=\frac{X^2}{1+\log X}+\mathcal{O}\Big(X^2\exp\big(-(\log X)^{1/3-\varepsilon}\big)\Big)
\end{equation}
holds.
\end{theorem}
As usual the corresponding binary problem is out of reach of the current state of analytic number theory.
In other words we have the following challenge.
\begin{conjecture} Let $N$ is a sufficiently large positive integer.
Then the logarithmic equation
\begin{equation*}
N=[p_1\log p_1]+[p_2\log p_2]
\end{equation*}
is solvable in prime numbers $p_1,\,p_2$.
\end{conjecture}
Needless to say we believe that in the near future
we will see the solution of this binary logarithmic hypothesis.

\section{Notations}
\indent

The letter $p$  with or without subscript will always denote prime number.
We denote by $\Lambda(n)$ von Mangoldt's function.
Moreover $e(y)=e^{2\pi \imath y}$.
As usual $[t]$ and $\{t\}$ denote the integer part, respectively, the
fractional part of $t$.
We recall that $t=[t]+\{t\}$ and $\|t\|=\min(\{t\}_,1-\{t\})$.
By $\varepsilon$ we denote an arbitrary small positive constant, not the same in all appearances.
Let $N$ be a sufficiently large positive integer.
Let $X$ is a solution of the equality
\begin{equation}\label{XlogXN}
X\log X=N.
\end{equation}
Let $y$ be an implicit function of $t$ defined by
\begin{equation}\label{ylogyt}
y\log y=t.
\end{equation}
The first derivative of $y$ is
\begin{equation}\label{y'}
y'=\frac{1}{1+\log y}.
\end{equation}
Denote
\begin{align}
\label{tau}
&\tau=X^{-\frac{23}{25}}\,;\\
\label{Salpha}
&S(\alpha)=\sum\limits_{p\leq X} e\big(\alpha [p\log p]\big)\log p\,;\\
\label{Theta}
&\Theta(\alpha)=\sum\limits_{m\leq N}\frac{1}{1+\log y(m)}\, e(m\alpha)\,;\\
\label{Gamma1}
&\Gamma_1=\int\limits_{-\tau}^{\tau}S^3(\alpha)e(-N\alpha)\,d\alpha\,;\\
\label{Gamma2}
&\Gamma_2=\int\limits_{\tau}^{1-\tau}S^3(\alpha)e(-N\alpha)\,d\alpha\,;\\
\label{Psik}
&\Psi_k=\int\limits_{-1/2}^{1/2}\Theta^k(\alpha)e(-N\alpha)\,d\alpha, \quad k=1,\, 2,\, 3, \ldots\,;\\
\label{Psiwidetilde}
&\widetilde{\Psi}=\int\limits_{-\tau}^{\tau}\Theta^3(\alpha)e(-N\alpha)\,d\alpha\,.
\end{align}

\section{Lemmas}
\indent

\begin{lemma}\label{ThetaS}
Let $f(x)$ be a real differentiable function in the interval $[a,b]$.
If $f'(x)$ is a monotonous and satisfies $|f'(x)|\leq\theta<1$.
Then we have
\begin{equation*}
\sum_{a<n\le b}e(f(n))=\int\limits_{a}^{b}e(f(x))\,dx+\mathcal{O}(1)\,.
\end{equation*}
\end{lemma}
\begin{proof}
See (\cite{Titchmarsh},  Lemma 4.8).
\end{proof}

\begin{lemma}\label{Buriev} Let $x,y\in\mathbb{R}$ and $H\geq3$.
Then the formula
\begin{equation*}
e(-x\{y\})=\sum\limits_{|h|\leq H}c_h(x)e(hy)+\mathcal{O}\left(\min\left(1, \frac{1}{H\|y\|}\right)\right)
\end{equation*}
holds. Here
\begin{equation*}
c_h(x)=\frac{1-e(-x)}{2\pi i(h+x)}\,.
\end{equation*}
\end{lemma}
\begin{proof}
See (\cite{Buriev}, Lemma 12).
\end{proof}

\begin{lemma}\label{Korput} (Van der Corput) Let $f(x)$ be a real-valued function with
continuous second derivative in $[a, b]$ such that
\begin{equation*}
|f''(x)|\asymp\lambda, \quad (\lambda>0)\quad  \mbox{ for } x\in[a,b].
\end{equation*}
Then
\begin{equation*}
\bigg|\sum_{a<n\le b}e(f(n))\bigg|
\ll(b-a)\lambda^{\frac{1}{2}}+\lambda^{-\frac{1}{2}}.
\end{equation*}
\end{lemma}
\begin{proof}
See (\cite{Karatsuba}, Ch. 1, Th. 5).
\end{proof}

\begin{lemma}\label{Expansion}
For any real number $t$ and $H\geq1$, there holds
\begin{equation*}
\min\left(1, \frac{1}{H\|t\|}\right)
=\sum\limits_{h=-\infty}^{+\infty}a_he(h t)\,,
\end{equation*}
where
\begin{equation*}
a_h\ll\min\left(\frac{\log 2H}{H}, \frac{1}{|h|}, \frac{H}{|h|^2}, \right).
\end{equation*}
\begin{proof}
See (\cite{Zhang-Li1}, Lemma 2).
\end{proof}
\end{lemma}

\section{Proof of the Theorem}
\indent

From \eqref{Gamma}, \eqref{Salpha},  \eqref{Gamma1}  and \eqref{Gamma2} we have
\begin{equation}\label{GammaGamma12}
\Gamma=\int\limits_{0}^{1}S^3(\alpha)e(-N\alpha)\,d\alpha=\Gamma_1+\Gamma_2.
\end{equation}

\textbf{Estimation of $\Gamma_1$}

We write
\begin{equation}\label{Gamma1decomp}
\Gamma_1=(\Gamma_1-\widetilde{\Psi})+(\widetilde{\Psi}-\Psi_3)+\Psi_3.
\end{equation}
Bearing in mind \eqref{Theta} and \eqref{Psik} we obtain
\begin{equation*}
\Psi_1=\int\limits_{-1/2}^{1/2}\Theta(\alpha)e(-N\alpha)\,d\alpha
=\frac{1}{1+\log y(N)}.
\end{equation*}
Suppose that
\begin{equation}\label{Supposition}
\Psi_k=\frac{1}{1+\log y(N)}X^{k-1}+\mathcal{O}\big(X^{k-2}\big) \quad \mbox{for} \quad k\geq2.
\end{equation}
Then
\begin{align*}
\Psi_{k+1}&=\sum\limits_{m\leq N} \frac{1}{1+\log y(m)}
\Bigg(\mathop{\sum\limits_{m_1\leq N-m}\cdots\sum\limits_{m_k\leq N-m}}_
{m_1+\cdots+m_k=N-m}\frac{1}{1+\log y(m_1)}\cdots\frac{1}{1+\log y(m_k)}\Bigg)\\
&=\sum\limits_{m\leq N} \frac{1}{1+\log y(m)}\Bigg(\frac{1}{1+\log y(N-m)}X^{k-1}+\mathcal{O}\big(X^{k-2}\big)\Bigg)\\
&=\sum\limits_{m\leq N} \frac{1}{1+\log y(m)}\cdot\frac{1}{1+\log y(N-m)}X^{k-1}+\mathcal{O}\big(X^{k-1}\big)\\
&= \frac{1}{1+\log y(N)}X^k+\mathcal{O}\big(X^{k-1}\big).
\end{align*}
Consequently the supposition \eqref{Supposition} is true.\\
From \eqref{XlogXN} and \eqref{ylogyt} it follows that
\begin{equation}\label{yNX}
y(N)=X.
\end{equation}
Bearing in mind \eqref{Supposition} and \eqref{yNX} we conclude that
\begin{equation}\label{PsikAsymptoticformula}
\Psi_k=\frac{X^{k-1}}{1+\log X}+\mathcal{O}\big(X^{k-2}\big) \quad \mbox{for} \quad k\geq2.
\end{equation}
Now the asymptotic formula \eqref{PsikAsymptoticformula} gives us
\begin{equation}\label{Psi3Asymptoticformula}
\Psi_3=\frac{X^2}{1+\log X}+\mathcal{O}\big(X\big).
\end{equation}
From \eqref{Gamma1} and \eqref{Psiwidetilde} we get
\begin{align}\label{Gamma1Psi1}
|\Gamma_1-\widetilde{\Psi}|&\ll\int\limits_{-\tau}^{\tau}
\big|S^3(\alpha)-\Theta^3(\alpha)\big|\,d\alpha\nonumber\\
&\ll\max\limits_{|\alpha|\leq \tau}\big|S(\alpha)-\Theta(\alpha)\big|
\Bigg(\int\limits_{-\tau}^{\tau}|S(\alpha)|^2\,d\alpha
+\int\limits_{-1/2}^{1/2}|\Theta(\alpha)|^2\,d\alpha \Bigg).
\end{align}
Arguing as in (\cite{Dimitrov},  Lemma 8) we find
\begin{equation}\label{IntSalpha}
\int\limits_{-\tau}^{\tau}|S(\alpha)|^2\,d\alpha \ll X\log X.
\end{equation}
Square out and integrate we obtain
\begin{equation}\label{1IntTheta}
\int\limits_{-1/2}^{1/2}|\Theta(\alpha)|^2\,d\alpha
\ll \frac{N}{\log^2N} \ll X.
\end{equation}
Now we shall estimate from above $|S(\alpha)-\Theta(\alpha)|$  for $|\alpha|\leq \tau$.\\
Our argument is a modification of Zhang's and Li's  \cite{Zhang-Li2} argument.\\
From \eqref{tau} and \eqref{Salpha} we get
\begin{align}\label{Salphaest1}
S(\alpha)&=\sum\limits_{p\leq X}e(\alpha p\log p)\log p+\mathcal{O}\big(\tau X\big)
=\sum\limits_{n\leq X} \Lambda(n)e(\alpha n\log n)
+\mathcal{O}\big( X^{1/2}\big)+\mathcal{O}\big(\tau X\big)\nonumber\\
&=\sum\limits_{n\leq X} \Lambda(n)e(\alpha n\log n)+\mathcal{O}\big(X^{1/2}\big).
\end{align}
From $|\alpha|\leq \tau$, $y\geq2$  and Lemma \ref{ThetaS} we have that
\begin{equation}\label{ThetaSgives}
\sum_{1<m\le y}e(m\alpha)=\int\limits_{1}^{y}e(\alpha t)\,dt+\mathcal{O}(1).
\end{equation}
Using \eqref{ylogyt}, \eqref{y'}, \eqref{tau}, \eqref{Theta},
\eqref{ThetaSgives} and partial summation we find
\begin{align}\label{LambdaTheta}
\sum\limits_{n\leq X} \Lambda(n)e(\alpha n\log n)&
=\int\limits_{1}^{X}e(\alpha y\log y)\,d\bigg(\sum\limits_{n\leq y} \Lambda(n)\bigg) \nonumber \\
&=\int\limits_{1}^{X}e(\alpha y\log y)\,dy+\mathcal{O}\Big(X\exp\big(-(\log X)^{1/3}\big)\Big) \nonumber \\
&=\int\limits_1^Ne(\alpha t)
\frac{1}{1+\log y(t)}\,dt+\mathcal{O}\Big(X\exp\big(-(\log X)^{1/3}\big)\Big) \nonumber \\
&=\int\limits_1^N\frac{1}{1+\log y(t)}\,d\Bigg(\int\limits_1^te(\alpha u)\,du\Bigg)
+\mathcal{O}\Big(X\exp\big(-(\log X)^{1/3}\big)\Big)\nonumber \\
&=\int\limits_1^N\frac{1}{1+\log y(t)}
\,d\Bigg(\sum_{1<m\le t}e( m\alpha)+\mathcal{O}(1)\Bigg)\nonumber \\
&+\mathcal{O}\Big(X\exp\big(-(\log X)^{1/3}\big)\Big)\nonumber \\
&=\sum\limits_{ m\leq N}\frac{1}{1+\log y(m)}e(m\alpha)
+\mathcal{O}\Big(X\exp\big(-(\log X)^{1/3}\big)\Big)\nonumber \\
&=\Theta(\alpha)+\mathcal{O}\Big(X\exp\big(-(\log X)^{1/3}\big)\Big).
\end{align}
From \eqref{Salphaest1} and \eqref{LambdaTheta} it follows that
\begin{equation}\label{SalphaThetaalpha}
\max\limits_{|\alpha|\leq \tau}|S(\alpha)-\Theta(\alpha)|\ll X\exp\big(-(\log X)^{1/3}\big).
\end{equation}
Taking into account \eqref{Gamma1Psi1}, \eqref{IntSalpha}, \eqref{1IntTheta} and \eqref{SalphaThetaalpha}
we conclude
\begin{equation}\label{Gamma1widetildePsiest}
\Gamma_1-\widetilde{\Psi}\ll X^2\exp\big(-(\log X)^{1/3-\varepsilon}\big).
\end{equation}
Using \eqref{tau}, \eqref{Psik}, \eqref{Psiwidetilde} and working as in (\cite{Vaughan}, Lemma 2.8)  we deduce
\begin{equation}\label{Psi3-widetildePsi}
\big|\Psi_3-\widetilde{\Psi}\big|\ll\int\limits_{\tau\leq|\alpha|\leq1/2}
|\Theta(\alpha)|^3\,d\alpha\ll\int\limits_{\tau}^{1/2}\frac{d\alpha}{\alpha^3}
\ll X^{\frac{46}{25}}.
\end{equation}
Summarizing \eqref{Gamma1decomp}, \eqref{Psi3Asymptoticformula}, \eqref{Gamma1widetildePsiest}
and \eqref{Psi3-widetildePsi} we obtain
\begin{equation}\label{Gamma1asymptotic}
\Gamma_1=\frac{X^2}{1+\log X}+\mathcal{O}\Big(X^2\exp\big(-(\log X)^{1/3-\varepsilon}\big)\Big).
\end{equation}

\textbf{Estimation of $\Gamma_2$}

From \eqref{Gamma2} we get
\begin{equation}\label{Gamma2est1}
\Gamma_2\ll \max\limits_{\tau\leq\alpha\leq 1-\tau}|S(\alpha)|
\int\limits_{0}^{1}|S(\alpha)|^2\,d\alpha\ll
X(\log X)\max\limits_{\tau\leq\alpha\leq 1-\tau}|S(\alpha)|.
\end{equation}
By \eqref{Salpha} and Lemma \ref{Buriev} with $x=\alpha$, $y=n\log n$
and
\begin{equation}\label{H}
H=X^{\frac{1}{25}}
\end{equation}
it follows
\begin{align*}
S(\alpha)&=\sum\limits_{n\leq X}\Lambda(n)e(\alpha n\log n)e(-\alpha\{n\log n\})+\mathcal{O}( X^{1/2})\\
&=\sum\limits_{|h|\leq H}c_h(\alpha)\sum\limits_{n\leq X}\Lambda(n)
e\big((h+\alpha) n\log n\big)\\
&+\mathcal{O}\left((\log X)\sum\limits_{n\leq X}\min\left(1, \frac{1}{H\| n\log n\|}\right)\right).
\end{align*}
Therefore
\begin{equation}\label{Salphaest2}
\max\limits_{\tau\leq\alpha\leq 1-\tau}|S(\alpha)|\ll (S_1+S_2)\log X,
\end{equation}
where
\begin{align}
\label{S1}
&S_1=\max\limits_{\tau\leq\alpha\leq H+1}
\Big|\sum\limits_{n\leq X}\Lambda(n)e(\alpha n\log n)\Big|,\\
\label{S2}
&S_2=\sum\limits_{n\leq X}\min\left(1, \frac{1}{H\| n\log n\|}\right).
\end{align}
Bearing in mind \eqref{tau}, \eqref{H} and \eqref{S1}, according to (\cite{Dimitrov}, Lemma 9) we conclude
\begin{equation}\label{S1est}
S_1\ll X^{24/25}\log^3X.
\end{equation}
By \eqref{H}, \eqref{S2}, Lemma \ref{Korput}, Lemma \ref{Expansion} and $Y\leq X/2$  we obtain
\begin{align}\label{S2est}
S_2&\ll(\log X)\sum\limits_{Y<n\leq 2Y}\min\left(1, \frac{1}{H\| n\log n\|}\right)\nonumber\\
&\leq(\log X)\sum\limits_{h=-\infty}^{+\infty}|a_h|\bigg|
\sum\limits_{Y<n\leq 2Y}e(h  n\log n)\bigg|\nonumber\\
&\ll(\log X)\Bigg(\frac{Y\log 2H}{H}+\frac{Y^{1/2}\log 2H}{H}
\sum\limits_{h\leq H}h^{1/2}
+Y^{1/2}H\sum\limits_{h>H}h^{-3/2}\Bigg)\nonumber\\
&\ll XH^{-1}\log^2X\nonumber\\
&\ll X^{24/25}\log^2X.
\end{align}
From \eqref{Gamma2est1}, \eqref{Salphaest2}, \eqref{S1est} and \eqref{S2est} we find
\begin{equation}\label{Gamma2est2}
\Gamma_2\ll X^{49/25}\log^5X.
\end{equation}

\textbf{The end of the proof}

Bearing in mind \eqref{GammaGamma12}, \eqref{Gamma1asymptotic} and \eqref{Gamma2est2}
we establish the asymptotic formula \eqref{Asymptoticformula}.

The Theorem is proved.

\vskip18pt
\footnotesize
\begin{flushleft}
S. I. Dimitrov\\
Faculty of Applied Mathematics and Informatics\\
Technical University of Sofia \\
8, St.Kliment Ohridski Blvd. \\
1756 Sofia, BULGARIA\\
e-mail: sdimitrov@tu-sofia.bg\\
\end{flushleft}

\end{document}